\documentclass[12pt, leqno, twoside]{amsart}
\usepackage{amsmath, amsthm, amsfonts, amssymb, graphicx, color}
\numberwithin{equation}{section}
     \newtheorem{thm}{Theorem}[section]
     \newtheorem{cor}[thm]{Corollary}
     \newtheorem{prop}[thm]{Proposition}
     \newtheorem{lem}[thm]{Lemma}
\theoremstyle{definition}
      \newtheorem{defn}{Definition}[section]
     
\theoremstyle{remark}
     \newtheorem{rem}{Remark}[section]

\newcommand{\R}{\mathbb{R}}

\newcommand{\cL}{\mathcal{L}}
\newcommand{\cLN}{\mathcal{L}^{\natural}}

\newcommand{\LamN}{\Lambda^{\natural}}

\newcommand{\operator}{\mathrm{Op}}

\newcommand{\BMO}{\mathrm{BMO}}
\newcommand{\Lip}{\mathrm{Lip}}
\newcommand{\LipN}{\mathrm{Lip}^{\natural}}

\newcommand{\PWM}{\mathrm{PWM}}
\newcommand{\Linfty}{L^{\infty}}

\newcommand{\tphi}{\tilde\phi}
\newcommand{\tpsi}{\tilde\psi}

\newcommand{\tR}{\tilde{R}}

\newcommand{\ve}{\varepsilon}

\newcommand{\ax}{\alpha_{*}}

\newcommand{\pt}{\phantom{**}}

\def\be{\begin{equation}}
\def\ee{\end{equation}}

\begin{document}

\baselineskip=18pt

\title{
A blowup criteria along maximum points of the 3D-Navier-Stokes flow in terms of 
function spaces with variable growth condition
}

\author{Eiichi Nakai}
\address{%
Department of Mathematics, Ibaraki University,
Mito, Ibaraki 310-8512, Japan}
\email{enakai@mx.ibaraki.ac.jp}
\author{Tsuyoshi Yoneda}
\address{%
Department of Mathematics, Tokyo Institute of Technology, 
Meguro-ku, Tokyo 152-8551, Japan}
\email{yoneda@math.titech.ac.jp}

\begin{abstract}
\noindent
A blowup criteria along maximum point of the 3D-Navier-Stokes flow 
in terms of function spaces with variable growth condition is constructed. 
This criterion is different from the Beale-Kato-Majda type and Constantin-Fefferman type criterion.
If geometric behavior of the velocity vector field near the maximum point has a kind of symmetry
up to a possible blowup time, 
then the solution can be extended to be the strong solution beyond the possible blowup time. 
\end{abstract}

\maketitle

\bigskip


\bigskip

\section{Introduction}\label{s:intro}
\noindent
In this paper we construct a blowup criteria along maximum points of the 3D-Navier-Stokes flow in terms of 
function spaces with variable growth condition.
The Navier-Stokes
equation is expressed as 
\begin{equation}\label{NS}
\begin{cases}
 \partial_tv+(v\cdot \nabla)v-\Delta v+\nabla p=0
 & \text{in}\ \mathbb{R}^3\times[0,T),
\\
 \nabla\cdot v=0
 & \text{in}\ \mathbb{R}^3\times[0,T),
\\
 v_0=v|_{t=0}
 & \text{in}\ \mathbb{R}^3,
\end{cases}
\end{equation}
where $v$ is a vector field representing velocity of the fluid, and $p$ is the pressure. 
The most significant 
blowup criterion must be
the Beale-Kato-Majda  criterion \cite{BKM}. 
The Beale-Kato-Majda criterion is as follows:

\begin{thm}\label{thm:BKM}
Let 
$s>1/2$, and let $v_0\in H^s$ with $\text{div}\ v_0=0$ in distribution sense. Suppose that $v$ is a strong solution of \eqref{NS}.
If
\begin{equation}\label{BKM criterion}
\int_0^T\|\text{curl}\ v(t)\,\|_{\infty}dt<\infty,
\end{equation}
then $v$ can be extended to the strong solution up to some $T'$ with $T'>T$. 
\end{thm}
\noindent
This blowup criterion was further improved by
Giga~\cite{G}, Kozono and Taniuchi~\cite{KT}, the authors~\cite{NY}, etc. 
On the other hand, 
Constantin and Fefferman \cite{CF} (see also \cite{CFM})  
took into account geometric structure of the vortex stretching
term in the vorticity equations to get another kind of blowup condition. 
They imposed vortex direction condition to the high vorticity part.
This criterion was also further improved by, for example, 
Deng, Hou and Yu~\cite{DHY}.
These two separate forms of criteria controlling the blow-up by magnitude 
and the direction of the vorticity respectively are interpolated by Chae~\cite{Chae}.
For the detail of the blowup problem 
of the Navier-Stokes equation, see Fefferman~\cite{F} for example. 

In this paper, we give a different type of blowup criterion from them. 
We focus on a geometric behavior of the velocity vector field near 
the each maximum points. 
In order to state our blowup criterion, 
we need to give several definitions.

Let us denote a maximum point of $|v|$ at a time $t$ 
as $x_M=x_{M(t)}\in\R^3$ 
(if there are several maximum points at a time $t$, 
then we choose one maximum point. 
We sometimes abbreviate the time $t$).
We use rotation and transformation and bring a maximum point to 
the origin and its direction parallel to $x_3$-axis.
Then we decompose $v$ into two parts: 
symmetric flow  part  and its remainder. In this paper we prove that, if the remainder part is small, then the solution never blowup.

Let us explain precisely. We denote the unit tangent vector as
\begin{equation*}
 \tau(x_M)=\tau(x_{M(t)})=(v/|v|)(x_{M(t)},t),
\end{equation*}
and we choose unit normal vectors $n_1(x_M)$ and $n_2(x_M)$ as 
\begin{equation*}
\tau(x_M)\cdot n_1(x_M)=\tau(x_M)\cdot n_2(x_M)=n_1(x_M)\cdot n_2(x_M)=0.
\end{equation*}
Note that $n_1$ and $n_2$ are not uniquely determined.
We now construct a Cartesian coordinate system 
with a new $y_1$-axis to be the straight line which passes 
through the maximum point and is parallel to $n_1$, 
and a new $y_2$-axis to be the straight line which passes 
through the maximum point and is parallel to $n_2$. 
We set $y_3$-axis by $\tau$ in the same process.
Here we fix the maximum point $x_M=x_{M(t_*)}$ at $t=t_*$ for some time.
Then $v$ can be expressed as 
\begin{equation}\label{v to u}
v(x,t)=\tilde u_1(x,t)n_1(x_{M(t_*)})+\tilde u_2(x,t)n_2(x_{M(t_*)})+\tilde u_3(x,t)\tau(x_{M(t_*)}),
\end{equation}
with $\tilde u=(\tilde u_1,\tilde u_2,\tilde u_3)$, 
where
\begin{align*}
 \tilde u_1(x,t)
 &=v(x,t)\cdot n_1(x_{M(t_*)}), \\
 \tilde u_2(x,t)
 &=v(x,t)\cdot n_2(x_{M(t_*)}), \\
 \tilde u_3(x,t)
 &=v(x,t)\cdot \tau(x_{M(t_*)}). 
\end{align*}
Let $y=(y_1,y_2,y_3)$ be the coordinate representation of the point $x$ 
in the coordinate system based at the maximum point 
which is specified by the orthogonal frame $\{n_1,n_2,\tau\}$. 
That is, the point $x\in\R^3$ can be realized as 
$x=x_{M}+n_1(x_M)y_1+n_2(x_M)y_2+\tau(x_M)y_3$ with $x_M=x_{M(t_*)}$.
Then we can rewrite $\tilde u(x)=\tilde u(x,t)$ 
to $u(y)=u(y,t)=u_{M(t_*)}(y,t)$ as 
\begin{align*}
 u_1(y)&=u_1(y,t)=\tilde u_1(x_M+n_1(x_M)y_1+n_2(x_M)y_2+\tau(x_M)y_3,t),\\
 u_2(y)&=u_2(y,t)=\tilde u_2(x_M+n_1(x_M)y_1+n_2(x_M)y_2+\tau(x_M)y_3,t),\\
 u_3(y)&=u_3(y,t)=\tilde u_3(x_M+n_1(x_M)y_1+n_2(x_M)y_2+\tau(x_M)y_3,t).
\end{align*}
In this case $u_1(0,t_*)=u_2(0,t_*)=0$ and $u_3(0,t_*)=|v(x_{M(t_*)},t_*)|$.

Since the Navier-Stokes equation is rotation and translation invariant,  
$u$ also satisfies the  Navier-Stokes equation \eqref{NS} in $y$-valuable. 
Then $\nabla p$, in $y$-valuable, can be expressed as 
\begin{equation*}
 \nabla p
 =
 \sum_{i,j=1}^3R_iR_j\nabla(u_iu_j),
\end{equation*}
where $R_j$ ($j=1,2,3$) are the Riesz transforms.
We decompose $u$ into two parts; 
symmetric flow part $U$ and its remainder part $r$:
$$
 u=U+r.
$$
The symmetric flow part $U$ can be defined as follows:

\begin{defn}\label{symmetric flow} 
We say $U$ is a symmetric flow if $U$ satisfies
\begin{equation*}
\begin{cases}
U_1(y_1,y_2,y_3)=-U_1(y_1,y_2,-y_3),\\
U_2(y_1,y_2,y_3)=-U_2(y_1,y_2,-y_3),\\
U_3(y_1,y_2,y_3)=U_3(y_1,y_2,-y_3).
\end{cases}
\end{equation*}
\end{defn}

We see that the symmetric flow cannot create large gradient of the pressure.
Actually, a basic calculation shows that
\begin{equation}\label{pressure zero}
 \sum_{i,j=1}^3R_iR_j\partial_{3}(U_iU_j)|_{y=0}=0,
\end{equation}
since, if $f$ is even (odd) with respect to $y_3$,
then $R_1f$ and $R_2f$ are also even (odd) with respect to $y_3$,
but $R_3f$ is odd (even) with respect to $y_3$.
Thus we need to see the remainder part~$r$, 
namely, we have the following pressure formula:  
\begin{equation}\label{remainder estimate}
 \partial_3 p|_{y=0}
 =
 \sum_{i,j=1}^3R_iR_j\partial_3\left(r_iU_j+U_ir_j+r_ir_j\right)|_{y=0}.
\end{equation}

In this paper, using the above formula, we construct 
a different type 
(from Beale-Kato-Majda type and Constantin-Fefferman type) 
of blowup criterion. 
We measure symmetricity of the 
flow near each maximum points by controlling the remainder part~$r$.
In order to obtain a reasonable blowup condition 
from \eqref{remainder estimate},
we need two function spaces $V=(V,\|\cdot\|_V)$ and $W=(W,\|\cdot\|_W)$ 
on $\R^3$ such that 
\begin{eqnarray}\label{ineq1}
|f(0)|&\leq& \|f\|_{W}, \\
\label{ineq2}
\|R_iR_jf\|_{W}&\leq&C \|f\|_{W}, \\
\label{ineq3}
\|fg\|_{W}&\leq& C\|f\|_{V}\|g\|_{V}.
\end{eqnarray}
That is,
we need some smoothness condition at the origin for functions in $W$,
the boundedness of Riesz transforms on $W$
and the boundedness of pointwise multiplication operator as $V\times V\to W$.
Moreover, it is known that
there exist positive constants $R$ and $C$ such that 
\begin{equation}\label{C/r}
|v(x,t)|\leq C/|x|\quad\text{for}\quad |x|>R,
\end{equation}
where $R$ and $C$ are independent of $t\in[0,T)$. 
This is due to Corollary 1 in \cite{CKN} 
(we use the partial regularity result to the decay). 
See also Section 1 in \cite{CY}.
We need to take the decay condition \eqref{C/r} into account to construct $V$.
In these points of view, 
we use Campanato spaces with variable growth condition.
We discuss these function spaces in Sections \ref{s:space}--\ref{s:exmp}.
The following definition is the key in this paper.

\begin{defn}\label{no local collapsing}
We say ``$v$ is  no local collapsing (of its symmetricity near each maximum points)" 
with respect to the function space $V$,
if there exist constants $C>0$ and $\alpha<2$ such that, 
for each fixed $x_{M(t_*)}$ at $t_*\in[0,T)$,
$u=u_{M(t_*)}$ has the following property:
\begin{equation*}
 \inf_{u=U+r}
 \left\{\sum_{i,j}
 \left(
 \|\partial_3r_i\|_{V} \|U_j\|_{V}
  + \|r_i\|_{V} \|\partial_3 U_j\|_{V}
  + \|r_i\|_{V} \|\partial_3 r_j\|_{V}
 \right) \bigg|_{t=t_*}
 \right\}
 \leq
 C\frac{(T-t_*)^{-\alpha}}{u_3(0,t_*)},
\end{equation*}
where the infimum is taken over all decomposition $u=U+r$ with symmetric flow $U$.
\end{defn}

Roughly saying, if $\|\partial_3 r_j\|_V$ and $\|r_j\|_V$ 
are sufficiently small compare to $\|\partial_3 U_j\|_V$ and $\|U_j\|_V$
(which means symmetric part is dominant), then $v$ is no local collapsing.

The following is the main theorem.
\begin{thm}[Blowup criteria along maximum points]\label{theorem 1}
Let function spaces $V$ and $W$ satisfy 
\eqref{ineq1}, \eqref{ineq2} and \eqref{ineq3}.
Let $v_0$ be any non zero, smooth, divergence-free vector field 
in Schwartz class, that is,
\begin{equation*}
 |\partial_x^\alpha v_0(x)|
 \leq
 C_{\alpha,K} (1+|x|)^{-K} \quad\text{in}\quad \mathbb{R}^3
\end{equation*} 
for any $\alpha\in\mathbb{Z}_+^3$ and any $K>0$.
Suppose that $v\in C^\infty([0,T)\times\mathbb{R}^3)$
is a unique smooth solution of \eqref{NS} up to $T$.
If  $v$  is no local collapsing with respect to $V$, 
then $v$ can be extended to the strong solution up to some $T'$ with $T'>T$. 
\end{thm}

In the next section we prove Theorem~\ref{theorem 1}
by using the regularity criterion by \cite{G}.
We also give an example of function with no local collapsing
which doesn't satisfy 
the Beale-Kato-Majda criterion.
In Section~\ref{s:space} we define 
Campanato spaces with variable growth condition
which give concrete function spaces $V$ and $W$ 
satisfying \eqref{ineq1}, \eqref{ineq2} and \eqref{ineq3}.
Campanato spaces with variable growth condition 
were introduced by \cite{NakaiYabuta1985JMSJ}
to characterize the pointwise multipliers on $\BMO$,
and then they were investigated 
by \cite{Nakai1997Studia,Nakai2006Studia,Nakai2010RMC,NakaiSawano2012JFA},
etc.
Roughly saying, the function spaces $V$ and $W$ are required to
express $C^\alpha$ ($0<\alpha<1$) continuity near the origin
and the decay condition \eqref{C/r} far from the origin.
For these requirement, we can use 
Campanato spaces 
with variable growth condition. 
We state the boundedness of the Riesz transforms 
and the pointwise multiplication operator on these function spaces
in Section~\ref{s:SIO} and Section~\ref{s:PWM},
respectively.
Finally, we show that Campanato spaces satisfy the conditions
\eqref{ineq1}, \eqref{ineq2} and \eqref{ineq3} for some 
variable growth condition 
in Section~\ref{s:exmp}.

\section{Proof of the main theorem}

In this section we give a proof of the main theorem.
First we show a lemma.

\begin{lem}\label{lem}
Under the assumption of Theorem~\ref{theorem 1},
for each fixed $x_{M(t_*)}$, 
the following inequalities hold:
\begin{align}\label{universal constant}
 -(v\cdot\nabla p)(x_{M(t_*)},t_*)
 &\leq
 C (T-t_*)^{-\alpha},
\\
 (v\cdot \Delta v)(x_{M(t_*)},t_*)
 &\leq 0.
\label{Delta}
\end{align}
\end{lem}

\begin{proof}
Using the derivative $\partial_3$ along $\tau$ direction,
we have
\begin{equation*}
 -(v\cdot\nabla p)(x_{M(t_*)},t_*)
 =
 -(u_3\partial_{3}p)(0,t_*),
\end{equation*}
since $u_1(0,t_*)=u_2(0,t_*)=0$.
Then, by \eqref{remainder estimate}, \eqref{ineq1} and the definition of no local collapsingness,
we get \eqref{universal constant}.

Next we show \eqref{Delta}.
To do this we prove
\begin{equation*}
 (u_3\Delta u_3)(0,t_*)\le 0,
\end{equation*}
where $\Delta$ is the Laplacian with respect to $y=(y_1,y_2,y_3)$.
Since $y=0$ is a maximum point, we see 
\begin{equation*}
\partial_{j}|u(y)|\bigg|_{y=0}=0\quad\text{for}\quad j=1,2,3,
\end{equation*}
and 
\begin{equation*}
\partial_{j}^2|u(y)|\bigg|_{y=0}\leq  0\quad\text{for}\quad j=1,2,3.
\end{equation*}
There are smooth functions $\theta_1$, $\theta_2$ and $\theta_3$ such that 
\begin{eqnarray*}
u_1(y)&=&|u(y)|\sin\theta_1(y), \\
u_2(y)&=&|u(y)|\sin\theta_2(y), \\
u_3(y)&=&|u(y)|\cos\theta_3(y)
\end{eqnarray*}
with  $\theta_1(0)=\theta_2(0)=\theta_3(0)=0$.
A direct calculation yields
\begin{eqnarray*}
\partial_{1}u_3(y)&=&\partial_{1}|u(y)|\cos \theta_3(y)-|u(y)|\sin\theta_3(y)\partial_{1}\theta_3(y)\\
\partial_{1}^2u_3(y)&=&\partial_{1}^2|u(y)|\cos \theta_3(y)-2\partial_{1}|u(y)|\sin\theta_3(y)\partial_{1}\theta_3(y)\\
& &\phantom{*}
 -|u(y)|\cos\theta_3(y)(\partial_{1}\theta_3(y))^2-|u(y)|\sin\theta_3(y)\partial_{1}^2\theta_3(y).
\end{eqnarray*}
Thus we have 
\begin{equation*}
\partial_{1}^2u_3(y)\bigg|_{y=0}=\partial_{y_1}^2|u(y)|-|u(y)|(\partial_{y_1}\theta_3(y))^2\bigg|_{y=0}\leq 0.
\end{equation*}
Similar calculations to $y_2$ and $y_3$ directions, we have $(u_3\Delta u_3)(0,t_*)\le 0$.
\end{proof}

Next we define ``trajectory" $\gamma:[\tilde t,T)\to\mathbb{R}^3$ starting at a point $\tilde x$:
\begin{equation*}
\partial_t\gamma(\tilde x,\tilde t;t)=v( \gamma(\tilde x,\tilde t; t), t)
\quad\text{with}\quad \gamma(\tilde x,\tilde t;\tilde t)=\tilde x.
\end{equation*}
Then $\gamma$ provides a diffeomorphism
and the equation \eqref{NS} can be rewritten as follows:
\begin{equation*}
 \partial_t\bigg(v(\gamma(\tilde x,\tilde t;t),t)\bigg)
 =(\Delta v-\nabla p)(\gamma(\tilde x,\tilde t;t),t)
 \quad (\tilde t<t<T)
\end{equation*}
with $\gamma(\tilde x,\tilde t; \tilde t)=\tilde x\in \mathbb{R}^3$.
Since $v$ is bounded for fixed $t\in [0,T)$,  
we can define $X(t)\subset\mathbb{R}^3$ 
as the set of all maximum points of $|v(\cdot,t)|$ at a time $t\in[0,T)$,
namely,
\begin{equation*}
 |v(x,t)|
 =
 \sup_{\xi\in\R^3}|v(\xi,t)|\ \text{for}\ x\in X(t)
 \quad\text{and}\quad
 |v(x,t)|
 <
 \sup_{\xi\in\R^3}|v(\xi,t)|\ \text{for}\ x\not\in X(t).
\end{equation*}
By \eqref{C/r}, $X(t)$ is a bounded set uniformly in $t$
in a possible blowup scenario.
Let $B(x,r)$ is a ball with radius $r$ and centered at $x$. 
For any $r>0$, we see that there is a barrier function  $\beta(t)>0$ such that 
\begin{equation*}
 |v(x,t)|+\beta(t)
 <
 \sup_{\xi\in\R^3}|v(\xi,t)|
 \quad \text{for}\quad 
 x\not\in \cup_{\xi\in X(t)}B(\xi,r).
\end{equation*}


Then, using Lemma~\ref{lem} and the smoothness of the solution, we get the following:

\begin{prop}\label{prop}
Under the assumption of Theorem~\ref{theorem 1},
for any $\delta>0$ and $t_*\in[0,T)$,
there exists a time interval $[t_*,t_*')\subset[0,T)$ and a radius $r_*$ 
such that the following two properties hold for all $t'\in[t_*,t_*')$:
\begin{itemize}
\item
$
\cup_{\xi\in X(t_*)}B(\xi,r_*)
\Subset \Omega(t'),
$
where 
\begin{multline}\label{Omega}
 \Omega(t'):=\bigg\{x\in\mathbb{R}^3:
 (\Delta v\cdot v)(\gamma(x,t_*;t'),t')\leq \delta,\\ 
(-\nabla p\cdot v)(\gamma(x,t_*;t'),t')\leq \delta+C(T-t')^{-\alpha}\bigg\},
\end{multline} 
\item
$
|v(\gamma(x,t_*;t'), t')|^2<
\sup_{\xi\in\mathbb{R}^3}|v(\xi,t_*)|^2
\quad\text{for}\quad x\in \left(\cup_{\xi\in X(t_*)}B(\xi,r_*)\right)^c$.
\end{itemize}
\end{prop}

\begin{proof}[Proof of Theorem~\ref{theorem 1}]
Note that the open interval $(0,T)$ is covered by the collection $\{(t_*,t_*')\}_{t_*\in[0,T)}$ of the open intervals
such that the interval $[t_*,t_*')$ is as in Proposition~\ref{prop} for $t_*\in[0,T)$.
Since $(0,T)$ is a Lindel\"of space,
we can choose a sequence of the time intervals $[t_j,t_j')$, $j=0,1,2,\cdots$ (finite or infinite),
such that $(0,T)=\cup_{j}(t_j,t_j')$,
and that $[t_j,t_j')$ and $r_j$ satisfy the properties of Proposition~\ref{prop} for $t_j\in[0,T)$.
We may assume that 
\begin{equation*}
 0=t_0<t_1<t_2<\cdots,
 \quad
 t_{j+1}<t_j',\ j=0,1,\cdots
\end{equation*} 
For $t\in[t_0,t_0')$ and $x\in\cup_{\xi\in X(t_0)}B(\xi,r_0)$, 
from the first property in Proposition~\ref{prop} it follows that
\begin{eqnarray*}
& &|v(\gamma(x,t_0;t),t)|^2\\
&=&\int_{t_0}^t\partial_{t'}|v(\gamma(x,t_0;t'),t')|^2dt'+|v(x,t_0)|^2 \\
&=& 2\int_{t_0}^t\partial_{t'} v\cdot v dt'+|v(x,0)|^2\\
&=& 2\int_{t_0}^t\left(\Delta v\cdot v-\nabla p\cdot v\right) dt'+|v(x,t_0)|^2\\
&\leq &
2\left(2\delta(t-t_0)+C\int_{t_0}^t(T-t')^{-\alpha}dt'\right)
 +\sup_{\xi\in\mathbb{R}^3}|v(\xi,t_0)|^2.
\end{eqnarray*}
The case $x\in(\cup_{\xi\in X(t_0)}B(\xi,r_0))^c$ 
is straightforward by the second property in Proposition~\ref{prop}.
Then we have 
\begin{equation*}
|v({z},t)|^2
\leq 2\left(2\delta(t-t_0)+C\int_{t_0}^t(T-t')^{-\alpha}dt'\right)
 +\sup_{\xi\in\mathbb{R}^3}|v(\xi,t_0)|^2.
\end{equation*}
for all $t\in[t_0,t_0')$ and all $z\in\mathbb{R}^3$ with $z=\gamma(x,t_0;t)$, 
since $\gamma$ gives a diffeomorphism.
Repeating the above  argument infinite times, and we finally have
\begin{equation*}
|v(x,t)|^2
\leq 2\left(2\delta t+C\int_0^t(T-t')^{-\alpha}dt'\right)+\sup_{\xi\in\mathbb{R}^3}|v(\xi,0)|^2
\end{equation*}
for all $t\in [0,T)$ and all $x\in\mathbb{R}^3$.
This implies
\begin{equation*}
 \|v\|_{L^2(0,T;L^\infty(\mathbb{R}^3))}<\infty.
\end{equation*} 
Due to the classical regularity criterion (see \cite{G} for example), 
we see that  the solution never blowup.
\end{proof}

\begin{rem}\label{counterexample}
We can  construct a function $u$  which satisfy both 
Definition~\ref{no local collapsing} and 
\begin{equation*}
 \int_0^T\|curl\ u(t)\|_\infty=\infty
 \quad\text{(the Beale-Kato-Majda criterion)}
\end{equation*}
(in this remark, $u$ is nothing to do with the Navier-Stokes solution, 
we just regard $u$ as a time dependent vector field).
If $\theta_j(y)=\theta_j(-y)$ ($j=1,2,3,$ even angular), we see that 
$\partial_3u_1(y)-\partial_1u_3(y)|_{y=0}$ is arrowed to be arbitrary large.
In fact,
\begin{align*}
 \partial_1u_3(y)
 &=
 (\partial_1|u(y)|)\cos\theta_3(y)
  -|u(y)|\sin\theta_3(y)\,\partial_1\theta_3(y),
\\
 \partial_3u_1(y)
 &=
 (\partial_3|u(y)|)\sin\theta_1(y)
  +|u(y)|\cos\theta_1(y)\,\partial_3\theta_1(y)
\end{align*}
and then 
\begin{equation*}
 \partial_3u_1(y)-\partial_1u_3(y)\bigg|_{y=0}
 =
 |u(y)|\partial_3\theta_1(y)\bigg|_{y=0}.
\end{equation*}
Since $\partial_3\theta_1(0)$ can be taken arbitrary large for each $t>0$, 
we can construct the desired function $u$.
Note that since $\theta_j(y)$ ($j=1,2,3$) are even angular, 
$u$ is symmetric flow (see Definition \ref{symmetric flow}). 
\end{rem}

\section{Campanato spaces with variable growth condition}\label{s:space}
In this section we define 
Campanato spaces $\cLN_{p,\phi}$ with variable growth condition.
We state basic properties of the function spaces 
$\cLN_{p,\phi}$. 
To do this we also define Morrey spaces and H\"older spaces 
with variable growth condition.

Let $\R^n$ be the $n$-dimensional Euclidean space.
We denote by $B(x,r)$ the open ball centered at $x\in\R^n$ and of radius $r$,
that is, 
\begin{equation*}
 B(x,r)
 =\{y\in\R^n:|y-x|<r\}.
\end{equation*}
For a measurable set $G \subset \R^n$, 
we denote by $|G|$ and $\chi_{G}$ 
the Lebesgue measure of $G$ and the characteristic function of $G$, 
respectively.

We consider variable growth functions 
$\phi:\R^n\times(0,\infty)\to(0,\infty)$.
For a ball $B=B(x,r)$, 
write $\phi(B)$ in place of $\phi(x,r)$.
For a function $f\in L^1_{\mathrm{loc}}(\R^n)$ and for a ball $B$, let
$$
 f_B = |B|^{-1} \int_B f(x) \,dx.
$$
Then 
we define Campanato spaces 
$\cL_{p,\phi}(\R^n)$ and $\cLN_{p,\phi}(\R^n)$, 
Morrey spaces $L_{p,\phi}(\R^n)$,
and H\"older spaces 
$\Lambda_{\phi}(\R^n)$ and $\LamN_{\phi}(\R^n)$
with variable growth functions $\phi$ as the following:

\begin{defn}\label{defn:CamMorHol}
For $1\le p <\infty$ and $\phi:\R^n\times(0,\infty)\to(0,\infty)$, 
function spaces 
$\cL_{p,\phi}(\R^n)$, $\cLN_{p,\phi}(\R^n)$, $L_{p,\phi}(\R^n)$,
$\Lambda_{\phi}(\R^n)$, $\LamN_{\phi}(\R^n)$
are the set of all functions $f$ such that
\begin{align*}
 \|f\|_{\cL_{p,\phi}} 
 &
 =\sup_B \frac 1{\phi(B)}\left( \frac 1{|B|} 
   \int_{B} |f(x)-f_{B}|^p \,dx \right)^{1/p}<\infty,
\\
 \|f\|_{\cLN_{p,\phi}}
 &=
 \|f\|_{\cL_{p,\phi}} + |f_{B(0,1)}|<\infty,
\\
 \|f\|_{L_{p,\phi}} 
 &
 =\sup_B \frac 1{\phi(B)}\left( \frac 1{|B|} 
  \int_{B} |f(x)|^p \,dx \right)^{1/p}<\infty,
\\
 \|f\|_{\Lambda_{\phi}} 
 &
 =\sup_{x,y\in \R^n, \; x\ne y} 
  \frac {2|f(x)-f(y)|}{\phi(x,|x-y|)+\phi(y,|y-x|)}<\infty,
\\
 \|f\|_{\LamN_{\phi}}
 &=
 \|f\|_{\Lambda_{\phi}} + |f(0)|<\infty,
\end{align*}
respectively.
\end{defn}

We regard 
$\cLN_{p,\phi}(\R^n)$ and
$L_{p,\phi}(\R^n)$
as spaces of functions modulo null-functions, 
$\cL_{p,\phi}(\R^n)$ 
as spaces of functions modulo null-functions and constant functions, 
$\LamN_{\phi}(\R^n)$ as a space of functions defined at all $x\in \R^n$,
and $\Lambda_{\phi}(\R^n)$ as a space of functions defined at all $x\in \R^n$
modulo constant functions.
Then 
these five functionals are norms 
and thereby these spaces are all Banach spaces.

In order to apply $\cLN_{p,\phi}$ to the blowup criterion 
(more precisely, in order to find specific function spaces $V$ and $W$ 
satisfying \eqref{ineq1}, \eqref{ineq2} and \eqref{ineq3}), 
we state several properties of these function spaces and relation between $\phi$ and the function spaces.
For two variable growth functions $\phi_1$ and $\phi_2$,
we write $\phi_1\sim\phi_2$ if
there exists a positive constant $C$ such that
$$
 C^{-1}\phi_1(B)\le\phi_2(B)\le C\phi_1(B)
 \quad\text{for all balls $B$}.
$$ 
In this case, two spaces defined by $\phi_1$ and by $\phi_2$ 
coincide with equivalent norms.
If $p=1$ and $\phi\equiv 1$, 
then $\cL_{p,\phi}(\R^n)$ is the usual $\BMO(\R^n)$.
For $\phi(x,r)=r^{\alpha}$, $0<\alpha\le1$, 
we denote 
$\Lambda_{r^{\alpha}}(\R^n)$ and $\LamN_{r^{\alpha}}(\R^n)$ 
by $\Lip_{\alpha}(\R^n)$ and $\LipN_{\alpha}(\R^n)$,
respectively.
In this case,
\begin{equation*}
 \|f\|_{\Lip_{\alpha}} 
 =
 \sup_{x,y\in \R^n, \; x\ne y} \frac {|f(x)-f(y)|}{|x-y|^{\alpha}}
 \quad\text{and}\quad
 \|f\|_{\LipN_{\alpha}} 
 =
 \|f\|_{\Lip_{\alpha}} + |f(0)|.
\end{equation*}
If $\phi(x,r)=\min(r^{\alpha},1)$, $0<\alpha\le1$, then
\begin{equation*}
 \|f\|_{\LamN_{\phi}} 
 \sim
 \|f\|_{\Lip_{\alpha}}+\|f\|_{\Linfty}.
\end{equation*}
From the definition it follows that
\begin{equation*}
 \|f\|_{\cL_{p,\phi}}\le 2\|f\|_{{L}_{p,\phi}},
 \quad
 \|f\|_{\cLN_{p,\phi}}\le (2+\phi(0,1))\|f\|_{{L}_{p,\phi}}.
\end{equation*}
If $\phi(B)=|B|^{-1/p}$ for all balls $B$, 
then 
\begin{equation*}
 \|f\|_{L_{p,\phi}}
 =
 \|f\|_{L^p}.
\end{equation*}

We consider the following conditions on variable growth function $\phi$:
\begin{alignat}{2}
  &\frac1{A_1} 
   \le \frac {\phi(x,s)}{\phi(x,r)} \le A_1, 
  & \quad \frac12\le \frac sr\le 2, 
                                                      \label{phi-double}\\
  &\frac1{A_2} 
   \le \frac {\phi(x,r)}{\phi(y,r)} \le A_2, 
  & \quad d(x,y)\le r, 
                                                      \label{phi-near} \\
  &\phi(x,r)
  \le A_3\phi(x,s),
  & \quad 0<r<s<\infty,
                                                      \label{phi-incr} 
\end{alignat}
where $A_i$, $i=1,2,3$, are positive constants independent of $x,y\in \R^n,\;r,s>0$.
Note that \eqref{phi-near} and \eqref{phi-incr} imply that
there exists a positive constant $C$ such that
$$
  \phi(x,r)\le C \phi(y,s)
  \quad \text{for}\quad 
  B(x,r) \subset B(y,s),
$$
where the constant $C$ is independent of balls $B(x,r)$ and $B(y,s)$.

The following three theorems are known:

\begin{thm}[{\cite{Nakai2008ActaMathSinica}}]\label{thm:Cam-p-1}
If $\phi$ satisfies \eqref{phi-double}, \eqref{phi-near} and \eqref{phi-incr},
then, for every $1\le p<\infty$,
$\cL_{p,\phi}(\R^n)=\cL_{1,\phi}(\R^n)$ 
and $\cLN_{p,\phi}(\R^n)=\cLN_{1,\phi}(\R^n)$ 
with equivalent norms, respectively.
\end{thm}

\begin{thm}[{\cite{Nakai2006Studia}}]\label{thm:Cam-Hol}
If $\phi$ satisfies \eqref{phi-double}, \eqref{phi-near}, \eqref{phi-incr},
and there exists a positive constant $C$ such that
\begin{equation}\label{int_0 phi}
  \int_0^r \frac{\phi(x,t)}{t}\,dt
  \le
  C\phi(x,r),
  \quad x\in \R^n,\ r>0,
\end{equation}
then, for every $1\le p<\infty$,
each element in $\cLN_{p,\phi}(\R^n)$ 
can be regarded as a continuous function,
(that is, each element is equivalent to a continuous function 
modulo null-functions)
and 
$\cL_{p,\phi}(\R^n)=\Lambda_{\phi}(\R^n)$ 
and $\cLN_{p,\phi}(\R^n)=\LamN_{\phi}(\R^n)$ 
with equivalent norms, respectively.
In particular, 
if $\phi(x,r)=r^{\alpha}$, $0<\alpha\le1$,  
then, for every $1\le p<\infty$,
$\cLN_{p,\phi}(\R^n)=\LipN_{\alpha}(\R^n)$
and
$\cL_{p,\phi}(\R^n)=\Lip_{\alpha}(\R^n)$
with equivalent norms, respectively.
\end{thm}

\begin{thm}[{\cite{Nakai2006Studia}}]\label{thm:Cam-Mor}
Let $1\le p<\infty$.
If $\phi$ satisfies \eqref{phi-double}, \eqref{phi-near}, 
and there exists a positive constant $C$ such that
\begin{equation}\label{int^infty phi}
 \int_r^{\infty}\frac{\phi(x,t)}t\,dt
  \le
  C\phi(x,r),
  \quad x\in \R^n,\ r>0,
\end{equation}
then,
for $f\in\cL_{p,\phi}(\R^n)$,
the limit
$
 \sigma(f)
 =
 \lim_{r\to\infty}f_{B(0,r)}
$
exists and
$$
 \|f\|_{\cL_{p,\phi}}
 \sim
 \|f-\sigma(f)\|_{L_{p,\phi}}.
$$
That is, the mapping $f\mapsto f-\sigma(f)$ is 
bijective and bicontinuous
from $\cL_{p,\phi}(\R^n)$ (modulo constants) 
to ${L}_{p,\phi}(\R^n)$.
\end{thm}

\begin{rem}\label{rem:Mor}
If $\int_1^{\infty}\phi(0,t)/t\,dt<\infty$, then
$\phi(0,r)\to0$ as $r\to\infty$.
Then, for $f\in L_{p,\phi}(\R^n)$,
we have
$$
 |\sigma(f)|
 =\lim_{r\to\infty}|f_{B(0,r)}|
 \le\lim_{r\to\infty}\phi(0,r)\|f\|_{L_{p,\phi}}
 \to0
 \quad\text{as}\quad 
 r\to\infty.
$$
That is, $\sigma(f)=0$.
\end{rem}

For a ball $B_*\subset\R^n$ and $0<\alpha\le1$, let
\begin{equation*}
 \|f\|_{\Lip_{\alpha}(B_*)}
 =
 \sup_{x,y\in B_*, \; x\ne y} 
 \frac {|f(x)-f(y)|}{|x-y|^{\alpha}}.
\end{equation*}
We also conclude the following:

\begin{prop}\label{prop:Cam-Lip-local}
Let $1\le p<\infty$ and $0<\alpha\le1$. 
Assume that,
for a ball $B_*$,
\begin{equation}\label{r^a in B}
 \phi(x,r)=r^{\alpha}
 \quad\text{for all balls $B(x,r)\subset B_*$}.
\end{equation}
Then
each element $f$ in $\cLN_{p,\phi}(\R^n)$ 
can be regarded as a continuous function
on the ball $B_*$,
and,
there exists a positive constant $C$ such that
\begin{equation*}
 \|f\|_{\Lip_{\alpha}(B_*)}
 \le
 C\|f\|_{\cL_{p,\phi}},
\end{equation*}
where $C$ is dependent only on $n$ and $\alpha$.
In particular, if \eqref{r^a in B} holds for $B_*=B(0,1)$,
then each $f\in\cLN_{p,\phi}(\R^n)$ is 
$\alpha$-Lipschitz continuous near the origin and
\begin{equation*}
 \|f\|_{\cLN_{p,\phi}}
 \sim
 \|f\|_{\cL_{p,\phi}}+|f(0)|.
\end{equation*}
\end{prop}

\begin{proof}
It is known that, if $\phi$ satisfies \eqref{phi-double}, then
\begin{equation}\label{fB1-fB2}
 |f_{B(x,r_1)}-f_{B(x,r_2)}|
 \le
 C\int_{r_1}^{2r_2}\frac{\phi(x,t)}t\,dt\ \|f\|_{\cL_{p,\phi}}
 \quad\text{for}\ x\in\R^n, \ r_1<r_2,
\end{equation} 
where $C$ is dependent only on $n$, 
see \cite[Lemma~2.4]{Nakai1993Studia}.
Hence
we have that, if $B(x,r)$, $B(y,r)\subset B_*$, then
\begin{equation*}
 |f_{B(x,r)}-f_{B(y,r)}|
 \le
 C\int_{r}^{2r+|x-y|}\frac{t^{\alpha}}t\,dt\ \|f\|_{\cL_{p,\phi}}
 \le 
 C_* (2r+|x-y|)^{\alpha}\ \|f\|_{\cL_{p,\phi}},
\end{equation*}
since $B(x,r)$, $B(y,r)\subset B((x+y)/2,r+|x-y|/2)$,
where $C_*$ is dependent only on $n$ and $\alpha$.
Letting $r\to0$, we have 
\begin{equation*}
 |f(x)-f(y)|
 \le
 C_* |x-y|^{\alpha}\ \|f\|_{\cL_{p,\phi}},
\end{equation*}
for almost every $x,y\in B_*$.
In this case we can regard that $f$ is a continuous function modulo null-functions
and we have
\begin{equation*}
 \|f\|_{\Lip_{\alpha}(B_*)}
 \le
 C_*\|f\|_{\cL_{p,\phi}}.
\end{equation*}
If $B_*=B(0,1)$, then
\begin{equation*}
 |f_{B(0,r)}-f_{B(0,1)}|
 \le
 C\int_{r}^{2}\frac{t^{\alpha}}t\,dt\ \|f\|_{\cL_{p,\phi}}
 \le 
 C \|f\|_{\cL_{p,\phi}}.
\end{equation*}
Letting $r\to0$, we have 
\begin{equation*}
 |f(0)-f_{B(0,1)}|
 \le
 C\|f\|_{\cL_{p,\phi}}.
\end{equation*}
This shows that $\|f\|_{\cL_{p,\phi}}+|f_{B(0,1)}|\sim\|f\|_{\cL_{p,\phi}}+|f(0)|$.
\end{proof}

\begin{prop}\label{prop:Cam-Lp-local}
Let $1\le p<\infty$ and $B_*$ be a ball such that $B(0,1)\subset B_*$.
Assume that 
there exists a positive constant $A$ such that
\begin{equation*}
 \phi(B)\le A|B|^{-1/p}
 \quad\text{for all balls $B\subset B_*$}.
\end{equation*}
Then
there exists a positive constant $C$ such that
\begin{equation*}
 \left(\int_{B_*}|f(x)|^p\,dx\right)^{1/p}
 \le
 C\|f\|_{\cLN_{p,\phi}}.
\end{equation*}
for all $f\in\cLN_{p,\phi}(\R^n)$, 
where $C$ is dependent only on $A$, $n$ and $p$.
\end{prop}

\begin{proof}
Let $B_*=B(x_*,r_*)$.
Using \eqref{fB1-fB2}, we have
\begin{equation*}
 |f_{B(0,1)}-f_{B(x_*,r_*)}|
 \le
 C\int_{1}^{2r_*}\frac{At^{-n/p}}t\,dt\ \|f\|_{\cL_{p,\phi}}
 \le 
 C_*\|f\|_{\cL_{p,\phi}},
\end{equation*}
where $C_*$ is dependent only on $A$, $n$ and $p$.
Then
\begin{align*}
 \left(\int_{B_*}|f(x)|^p\,dx\right)^{1/p}
 &\le
 \left(\int_{B_*}|f(x)-f_{B_*}|^p\,dx\right)^{1/p}
  + |f_{B(0,1)}-f_{B(x_*,r_*)}|
  + |f_{B(0,1)}|
\\
 &\le
 (A
+C_*)\|f\|_{\cL_{p,\phi}}+|f_{B(0,1)}|
\\
 &\le
 (A
+
C_*+1)\|f\|_{\cLN_{p,\phi}}.
\end{align*}
This shows the conclusion.
\end{proof}

\section{Singular integral operators}\label{s:SIO}

In this section we consider the singular integral theory to show 
the boundedness of Riesz transforms in 
Campanato spaces with variable growth condition.
We denote by $L^p_c(\R^n)$ 
the set of all $f\in L^p(\R^n)$ with compact support.
Let $0<\kappa\le1$.
We shall consider a singular integral operator $T$ 
with measurable kernel $K$ on $\R^n\times \R^n$ satisfying the following properties:
\begin{gather} 
  |K(x,y)|\le \frac{C}{|x-y|^n}
     \quad\text{for}\quad x\not=y,
                                                  \label{SK1}
 \\
  \begin{split}
  |K(x,y)-K(z,y)|+|K(y,x)-K(y,z)| 
  &\le 
  \frac{C}{|x-y|^n}
  \left(\frac{|x-z|}{|x-y|}\right)^{\kappa} \\
  &\text{for}\quad |x-y|\ge2|x-z|,
                                                  \label{SK2}
  \end{split}
 \\
  \begin{split}
  \int_{r\le|x-y|<R} K(x,y) \,dy
  =\int_{r\le|x-y|<R} K(y,x) \,dy = & \ 0  \\
  \text{for $0<r<R<\infty$ } & \text{and $x\in \R^n$},
                                                  \label{SK3}
  \end{split}
\end{gather}
where $C$ is a positive constant independent of $x,y,z\in \R^n$.
For $\eta>0$, let
\begin{equation*}
  T_{\eta}f(x)=\int_{|x-y|\ge\eta} K(x,y)f(y)\,dy.
\end{equation*}
Then $T_{\eta}f(x)$ is well defined for $f\in L^p_c(\R^n)$, $1<p<\infty$.
We assume that, for all $1<p<\infty$, 
there exists positive constant $C_p$ independently $\eta>0$ such that,
\begin{equation*} 
  \|T_{\eta}f\|_{L^p} \le C_p\|f\|_{L^p} \quad\text{for}\quad f\in L^p_c(\R^n),
\end{equation*}
and 
$T_{\eta}f$ converges to $Tf$ in $L^p(\R^n)$ as $\eta\to0$.
By this assumption, the operator $T$ can be extended as a continuous linear operator on $L^p(\R^n)$.
We shall say the operator $T$ satisfying the above conditions 
is a singular integral operator of type $\kappa$.
For example, Riesz transforms are singular integral operators of type $1$.

Now, to define $T$ for functions $f\in\cLN_{p,\phi}(\R^n)$,
we first define the modified version of $T_{\eta}$ by 
\begin{equation}\label{tildeT}
     {\tilde T}_{\eta}f(x)
       =\int_{|x-y|\ge\eta} f(y) 
          \big[K(x,y)-K(0,y)(1-\chi_{B(0,1)}(y))\big]
                 \,dy.
\end{equation}
Then we can show that the integral in the definition above 
converges absolutely for each $x$ 
and that
${\tilde T}_{\eta}f$ converges 
in $L^{p}(B)$ as $\eta\to0$ for each ball $B$. 
We denote the limit by ${\tilde T}f$. 
If both ${\tilde T}f$ and $Tf$ are well defined,
then the difference is a constant.

We can show the following results.
Theorem~\ref{thm:SI} is an extension of \cite[Theorem~4.1]{Nakai2010RMC}
and Theorem~\ref{thm:SI-Mor} is an extension of \cite[Theorem~2]{Nakai1994MathNachr}.
The proofs are almost the same.

\begin{thm}\label{thm:SI}
Let $0<\kappa\le1$ and $1< p<\infty$.
Assume that 
$\phi$ and $\psi$ satisfy \eqref{phi-double}
and that 
there exists a positive constant $A$ such that,
for all $x\in \R^n$ and $r>0$,
\begin{equation}\label{C1-A}
     r^{\kappa} \int_r^{\infty}\frac{\phi(x,t)}{t^{1+\kappa}}\,dt 
     \le A \psi(x,r).
\end{equation}
If $T$ is a singular integral operator of type $\kappa$,
then 
${\tilde T}$ is bounded 
from $\cL_{p,\phi}(\R^n)$
to $\cL_{p,\psi}(\R^n)$
and from $\cLN_{p,\phi}(\R^n)$
to $\cLN_{p,\psi}(\R^n)$,
that is, there exists a positive constants $C$ such that
$$
  \|\tilde{T}f\|_{\cL_{p,\psi}}
  \le C\|f\|_{\cL_{p,\phi}},
\quad
  \|\tilde{T}f\|_{\cLN_{p,\psi}}
  \le C\|f\|_{\cLN_{p,\phi}}.
$$
Moreover, if $\phi$ and $\psi$ satisfy \eqref{phi-near} and \eqref{phi-incr} also, 
then
${\tilde T}$ is bounded 
from $\cLN_{1,\phi}(\R^n)$ to $\cLN_{1,\psi}(\R^n)$.
\end{thm}

\begin{cor}\label{cor:SI Lam}
Under the assumption in Theorem~\ref{thm:SI},
if $\phi$ and $\psi$ satisfies 
\eqref{phi-near}, \eqref{phi-incr} and \eqref{int_0 phi},
then
${\tilde T}$ is bounded 
from $\Lambda_{\phi}(\R^n)$ to $\Lambda_{\psi}(\R^n)$
and from $\LamN_{\phi}(\R^n)$ to $\LamN_{\psi}(\R^n)$.
\end{cor}

For Morrey spaces $L_{p,\phi}(\R^n)$, we have the following.
\begin{thm}\label{thm:SI-Mor}
Let $0<\kappa\le1$ and $1< p<\infty$.
Assume that $\phi$ and $\psi$ satisfy \eqref{phi-double}
and that 
there exists a positive constant $A$ such that,
for all $x\in \R^n$ and $r>0$,
\begin{equation*}
     \int_r^{\infty}\frac{\phi(x,t)}{t}\,dt 
     \le A \psi(x,r).
\end{equation*}
If $T$ is a singular integral operator of type $\kappa$,
then 
$T$ is bounded from $L_{p,\phi}(\R^n)$ to $L_{p,\psi}(\R^n)$.
\end{thm}


Now we state the boundedness of Riesz transforms.
For $f$ in 
Schwartz class,
the Riesz transforms of $f$ are defined by
\begin{equation*}
 R_jf(x)
 =
 c_n\lim_{\ve\to0}R_{j,\ve}f(x),
     \quad j=1,\cdots, n,
\end{equation*}
where
\begin{equation*}
 R_{j,\ve}f(x)
 =
 \int_{\R^n\setminus B(x,\ve)}
 \frac{x_j-y_j}{|x-y|^{n+1}}f(y)\,dy,
 \quad
    c_n=\Gamma\left(\frac{n+1}{2}\right) \pi^{-\frac{n+1}{2}}.
\end{equation*}
Then it is known that
there exists a positive constant $C_p$ independently $\ve>0$ such that,
\begin{equation*} 
 \|R_{j,\ve}f\|_{L^p} \le C_p\|f\|_{L^p}
 \quad\text{for}\quad
 f\in L^p_c(\R^n),
\end{equation*}
and 
$R_{j,\ve}f$ converges to $R_jf$ in $L^p(\R^n)$ as $\ve\to0$.
That is, the operator $R_j$ can be extended as a continuous linear operator 
on $L^p(\R^n)$.
Hence, we can define a modified Riesz transforms of $f$ as 
\begin{equation*}
 \tR_jf(x)
 =
 c_n\lim_{\ve\to0}
 \tR_{j.\ve}f(x),
 \quad j=1,\cdots, n,
\end{equation*}
and
\begin{equation*}
 \tR_{j.\ve}f(x)
 =
 \int_{\R^n\setminus B(x,\ve)}
 \left(\frac{x_j-y_j}{|x-y|^{n+1}}
  -\frac{(-y_j)(1-\chi_{B(0,1)}(y))}{|y|^{n+1}}\right)f(y)\,dy.
\end{equation*}
We note that, if both $R_jf$ and $\tR_jf$ are well defined on $\R^n$, 
then $R_jf-\tR_jf$ is a constant function.
More precisely,
$$
 R_jf(x)-\tR_jf(x)
 =
 c_n\int_{\R^n}
  \frac{(-y_j)(1-\chi_{B(0,1)}(y))}{|y|^{n+1}}f(y)\,dy.
$$


\begin{rem}\label{rem:RT1}
If $f$ is a constant function, then $\tR_jf=0$.
Actually, for $f\equiv1$, 
\begin{align*}
 \tR_{j.\ve}1(x)
 &=
 \int_{\R^n\setminus B(x,\ve)}
 \frac{(x_j-y_j)\chi_{B(x,1)}}{|x-y|^{n+1}} \,dy 
\\
 &\pt +
 \int_{\R^n\setminus B(x,\ve)}
 \left(\frac{(x_j-y_j)(1-\chi_{B(x,1)})}{|x-y|^{n+1}}
  -\frac{(-y_j)(1-\chi_{B(0,1)}(y))}{|y|^{n+1}}\right)\,dy
\\
 &=
 \int_{B(0,1)\setminus B(0,\ve)}
 \frac{y_j}{|y|^{n+1}} \,dy
 +
 \int_{B(x,\ve)}
 \frac{(-y_j)(1-\chi_{B(0,1)}(y))}{|y|^{n+1}}\,dy
\\
 &=
 \int_{B(x,\ve)}
 \frac{(-y_j)(1-\chi_{B(0,1)}(y))}{|y|^{n+1}}\,dy
 \to0
 \quad\text{as $\ve\to0$},
\end{align*}
since
$$
 \int_{B(0,1)\setminus B(0,\ve)}
 \frac{y_j}{|y|^{n+1}} \,dy
 =0
$$
and
$$
 \int_{\R^n}
 \left(\frac{(x_j-y_j)(1-\chi_{B(x,1)})}{|x-y|^{n+1}}
  -\frac{(-y_j)(1-\chi_{B(0,1)}(y))}{|y|^{n+1}}\right)\,dy
 =0.
$$
Hence $\tR_j1(x)=0$ for all $x\in\R^3$.
\end{rem}

\begin{thm}\label{thm:Cam-RT}
Let $1\le p<\infty$ and
$\phi$ satisfy \eqref{phi-double} and 
$$
 r\int_r^{\infty}\frac{\phi(x,t)}{t^2}\,dt \le A\phi(x,r),
$$
for all $x\in\R^n$ and $r>0$.
Assume that
there exists a growth function $\tphi$ such that $\phi\le\tphi$
and that 
$\tphi$ satisfies 
\eqref{phi-double}, \eqref{phi-near} and \eqref{int^infty phi}.
If $f\in\cLN_{p,\phi}(\R^n)$ and $\sigma(f)=\lim_{r\to\infty}f_{B(0,r)}=0$,
then $R_jf$, $j=1,2,\cdots,n$, are well defined, 
$\sigma(R_jf)=\lim_{r\to\infty}(R_jf)_{B(0,r)}=0$, and
$$
 \|R_jf\|_{\cLN_{p,\phi}} \le C\|f\|_{\cLN_{p,\phi}},
 \quad
 j=1,2,\cdots,n,
$$
where $C$ is a positive constant independent of $f$.
\end{thm}

\begin{proof}
Let $f\in\cLN_{p,\phi}(\R^n)$ and $\sigma(f)=0$.
Then, by Theorem \ref{thm:Cam-Mor}, 
$$
 \|f\|_{L_{p,\tphi}}
 =\|f-\sigma(f)\|_{L_{p,\tphi}}
 \sim\|f\|_{\cL_{p,\tphi}}
 \le\|f\|_{\cL_{p,\phi}}
 \le\|f\|_{\cLN_{p,\phi}}.
$$
By Theorems~\ref{thm:SI-Mor} $R_jf$ is well defined and 
$$
 \|R_jf\|_{L_{p,\tphi}}
 \le C\|f\|_{L_{p,\tphi}}
 \le C\|f\|_{\cLN_{p,\phi}}.
$$
This shows that $\sigma(R_jf)=0$ by Remark~\ref{rem:Mor} and
$$
 |(R_jf)_{B(0,1)}|
 \le
 \left(\frac1{|B(0,1)|}\int_{B(0,1)}|R_jf(x)|^p\,dx\right)^{1/p}
 \le
 \tphi(0,1)\|R_jf\|_{L_{p,\tphi}}
 \le C\|f\|_{\cLN_{p,\phi}}.
$$
Since $R_jf-\tR_jf$ is a constant, by Theorem \ref{thm:SI}, we have
$$
 \|R_jf\|_{\cL_{p,\phi}}
 =\|\tR_jf\|_{\cL_{p,\phi}}
 \le C\|f\|_{\cL_{p,\phi}}
 \le C\|f\|_{\cLN_{p,\phi}}.
$$
Therefore, we have
$\|R_jf\|_{\cLN_{p,\phi}}\le C\|f\|_{\cLN_{p,\phi}}$.
\end{proof}

\section{Pointwise multiplication}\label{s:PWM}

Let $L^0(\R^n)$ be the set of all measurable functions on $\R^n$.
Let $X_1$ and $X_2$ be subspaces of $L^0(\R^n)$
and $g\in L^0(\R^n)$.
We say that $g$ is a pointwise multiplier from $X_1$ to $X_2$
if $fg\in X_2$ for all $f\in X_1$.
We denote by $\PWM(X_1,X_2)$ the set of all pointwise multipliers
from $X_1$ to $X_2$.

For $\phi:\R^n\times(0,\infty)\to(0,\infty)$, we define
\begin{align}\label{Phi*}
 \Phi^{*}(x,r)=\int_1^{\max(2,|x|,r)}\frac{\phi(0,t)}{t}\,dt,
\\
 \Phi^{**}(x,r)=\int_r^{\max(2,|x|,r)}\frac{\phi(x,t)}{t}\,dt. \label{Phi**}
\end{align}

\begin{prop}[{\cite[Proposition~4.4]{Nakai1997Studia}}]\label{prop:PWM}
Suppose that $\phi_1$ and $\phi_2$ satisfy the doubling condition \eqref{phi-double}.
For $\phi_1$, define
$\Phi_1^{*}$ and $\Phi_1^{**}$ by \eqref{Phi*} and \eqref{Phi**}, respectively.
Let $\phi_3=\phi_2/(\Phi_1^{*}+\Phi_1^{**})$.
If $1\le p_2<p_1<\infty$ and $p_4\ge p_1p_2/(p_1-p_2)$,
then
\begin{gather}
 \PWM(\cLN_{p_1,\phi_1}(\R^n),\cLN_{p_2,\phi_2}(\R^n))
 \supset
 \cLN_{p_2,\phi_3}(\R^n)\cap L_{p_4,\phi_2/\phi_1}(\R^n), 
\\
 \|g\|_{\operator}\le C(\|g\|_{\cL_{p_2,\phi_3}}+\|g\|_{L_{p_4,\phi_2/\phi_1}}),
\end{gather}
where $\|g\|_{\operator}$ is the operator norm of 
$g\in\PWM(\cLN_{p_1,\phi_1}(\R^n),\cLN_{p_2,\phi_2}(\R^n))$.
\end{prop}

\begin{lem}[{\cite[Lemma~3.5]{Nakai1997Studia}}]\label{lem:Cam-Mor}
Let $1\le p<\infty$.
Suppose that $\phi$ satisfies the doubling condition \eqref{phi-double}.
Then
\begin{equation}
 \cLN_{p,\phi}(\R^n)\subset L_{p,\Phi^{*}+\Phi^{**}}(\R^n)
 \quad\text{and}\quad
 \|f\|_{L_{p,\Phi^{*}+\Phi^{**}}}\le C\|f\|_{\cLN_{p,\phi}}.
\end{equation}
\end{lem}

\begin{cor}\label{cor:PWM}
Suppose that $\phi$ satisfies the doubling condition \eqref{phi-double}.
Let $\psi=\phi(\Phi^{*}+\Phi^{**})$.
If $1\le p_2<p_1<\infty$ and $p_4\ge p_1p_2/(p_1-p_2)$,
then
\begin{gather}
 \PWM(\cLN_{p_1,\phi}(\R^n),\cLN_{p_2,\psi}(\R^n))
 \supset
 \cLN_{p_4,\phi}(\R^n), 
\\
 \|g\|_{\operator}\le C\|g\|_{\cLN_{p_4,\phi}},
\end{gather}
where $\|g\|_{\operator}$ is the operator norm of 
$g\in\PWM(\cLN_{p_1,\phi}(\R^n),\cLN_{p_2,\psi}(\R^n))$.
This implies that
\begin{equation}
 \|fg\|_{\cLN_{p_2,\psi}}\le C\|f\|_{\cLN_{p_1,\phi}}\|g\|_{\cLN_{p_4,\phi}}.
\end{equation}
\end{cor}
For example, we can take $p_1=p_4=4$ and $p_2=2$.

\begin{proof}
By Lemma~\ref{lem:Cam-Mor} we have the inclusion
\begin{gather}
 \cLN_{p_2,\phi}(\R^n)\cap L_{p_4,\Phi^{*}+\Phi^{**}}(\R^n)
 \supset
 \cLN_{p_4,\phi}(\R^n),
\\
 \|g\|_{\cLN_{p_2,\phi}}+\|g\|_{L_{p_4,\Phi^{*}+\Phi^{**}}}
 \le
 C \|g\|_{\cLN_{p_4,\phi}}.
\end{gather}
Then, using Proposition~\ref{prop:PWM}, we have the conclusion.
\end{proof}


\section{Specific function spaces}\label{s:exmp}

We now give the specific function spaces $V$ and $W$ 
satisfying \eqref{ineq1}, \eqref{ineq2} and \eqref{ineq3}.

For example,
let $p>2$, $-n/p\le\ax<0<\alpha<1$, $-n/p\le\beta<0$, 
and
\begin{equation}\label{exmp1}
 \phi(x,r)=
 \begin{cases}
  r^{\alpha}, & |x|\le2,\  0<r\le2, \\
  r^{\beta}, & |x|\le2,\ r>2, \\
  r^{\ax}, & |x|>2,\ 0<r\le2, \\
  r^{\beta}, & |x|>2,\ r>2,
 \end{cases}
\quad
 \psi(x,r)=
 \begin{cases}
  r^{\alpha}, & |x|\le2,\  0<r\le2, \\
  r^{\beta}, & |x|\le2,\ r>2, \\
  r^{2\ax}, &  |x|>2,\ 0<r\le2, \\
  r^{\beta}, & |x|>2,\ r>2,
 \end{cases}
\end{equation}
and take
$$
 W=\cLN_{p/2,\psi}(\R^n) 
 \quad\text{and}\quad
 V=\cLN_{p,\phi}(\R^n),
$$
then $V$ and $W$ 
satisfy \eqref{ineq1}, \eqref{ineq2} and \eqref{ineq3} when $n=3$.
We will check these properties in this section.

Firstly, 
we see that $\phi$ and $\psi$ satisfy \eqref{phi-double} and
\begin{equation*}
 \psi(x,r)=r^{\alpha}
 \quad\text{for all $B(x,r)\subset B(0,2)$}.
\end{equation*}
Then, by Proposition~\ref{prop:Cam-Lip-local}, we have
\begin{equation*}
 \|f\|_{\Lip_{\alpha}(B(0,2))}
 \le C\|f\|_{\cL_{p/2,\psi}},
\end{equation*}
and 
$$
 \|f\|_{\cLN_{p/2,\psi}}
 \sim
 \|f\|_{\cL_{p/2,\psi}}+|f(0)|.
$$
This shows the property \eqref{ineq1}.
Next, the properties \eqref{ineq2} and \eqref{ineq3} 
follows from Propositions~\ref{prop:product} and \ref{prop:RT} below,
respectively.
Therefore, if $f,g\in\cLN_{p,\phi}(\R^n)$ and $\sigma(fg)=\lim_{r\to\infty}(fg)_{B(0,r)}=0$,
then 
\begin{equation*}
 |(R_jR_k(fg))(0)|
 \le
 \|R_jR_k(fg)\|_{\cLN_{p/2,\psi}}
 \le
 C\|fg\|_{\cLN_{p/2,\psi}}
 \le
 C\|f\|_{\cLN_{p,\phi}}\|g\|_{\cLN_{p,\phi}}.
\end{equation*}

Further,  
let $f$ be $\alpha$-Lipschitz continuous on $B(0,2)$ 
and $|f(x)|\le C/|x|$ for $|x|\ge2$.
Then $\sigma(f)=0$ and $f$ is in $\cLN_{p,\phi}(\R^n)$,
if $p$ and $\beta$ satisfy one of the following conditions:
\begin{equation*}
\begin{cases}
 2<p<n&\text{and}\ -1\le\beta<0,
\\
 p=n&\text{and}\ -1<\beta<0,
\\
 n<p&\text{and}\ -n/p\le\beta<0.
\end{cases}
\end{equation*}
Moreover,
if $\ax=\beta/2=-n/p$ also, then
$-n/(p/2)=2\ax=\beta<0$ and
\begin{equation*}
 \|R_jR_k(fg)\|_{\Lip_{\alpha}(B(0,2))}
 +\|R_jR_k(fg)\|_{L^{p/2}}
 \le
 C\|R_jR_k(fg)\|_{\cLN_{p/2,\psi}}
 \le
 C\|f\|_{\cLN_{p,\phi}}\|g\|_{\cLN_{p,\phi}},
\end{equation*}
for all $f,g\in\cLN_{p,\phi}(\R^n)$ satisfying $\sigma(fg)=0$,
see Proposition~\ref{prop:Cam-Lp-local}.

Note that, in the decomposition $u=U+r$ 
in Definition~\ref{no local collapsing}, 
we may assume that $U$ has a compact support in $\R^3$ at fixed $t$.
Then $|r(t,x)|\le C/|x|$ for large $x\in\R^3$.
It is also known that $\nabla u\in\Linfty(\R^3)$ at $t$,
see \cite{GIM}, that is, $\nabla r$ is bounded.
Hence $\sigma(\partial_3r_i U_j)=\sigma(r_i \partial_3U_j)
=\sigma(r_i \partial_3r_j)=0$ for all $i,j$.

\begin{prop}\label{prop:product}
Let $p\ge2$, $-n/p\le\ax<0<\alpha\le1$, $-n/p\le\beta<0$, 
and let $\phi$ and $\psi$ be as \eqref{exmp1}.
Then 
there exists a positive constant $C$ such that, 
for all $f,g\in\cLN_{p,\phi}(\R^n)$,
\begin{equation}
 \|fg\|_{\cLN_{p/2,\psi}}\le C\|f\|_{\cLN_{p,\phi}}\|g\|_{\cLN_{p,\phi}}.
\end{equation}
\end{prop}

\begin{proof}
For $\phi$ in \eqref{exmp1}, we have
\begin{equation*}
 \Phi^*(x,r)
 =
 \int_1^{\max(2,|x|,r)}\frac{\phi(0,t)}{t}\,dt
 =
 \int_1^2 t^{\alpha-1}\,dt+\int_2^{\max(2,|x|,r)} t^{\beta-1}\,dt
 \sim
 1,
\end{equation*}
and
\begin{align*}
 1+\Phi^{**}(x,r)
 &=
 1+\int_r^{\max(2,|x|,r)}\frac{\phi(x,t)}{t}\,dt
\\
 &=1+
 \begin{cases}
  \int_r^2 t^{\alpha-1}\,dt, & |x|\le2,\ 0<r\le2, \\
  0, & |x|\le2,\ r>2, \\
  \int_r^2 t^{\ax-1}\,dt+\int_2^{|x|} t^{\beta-1}\,dt, & |x|>2,\ 0<r\le2, \\
  \int_r^{\max(|x|,r)} t^{\beta-1}\,dt, & |x|>2,\ r>2,
 \end{cases}
\\
 &\sim
 \begin{cases}
  1, & |x|\le2,\  0<r\le2, \\
  r^{\ax}, & |x|>2,\ 0<r\le2, \\
  1, & r>2.
 \end{cases}
\end{align*}
Hence
\begin{equation*}
 \phi(x,r)(\Phi^*(x,r)+\Phi^{**}(x,r))
 \sim
 \psi(x,r)=
 \begin{cases}
  r^{\alpha}, & |x|\le2,\  0<r\le2, \\
  r^{2\ax}, & |x|>2,\ 0<r\le2, \\
  r^{\beta}, & r>2.
 \end{cases}
\end{equation*}
Then, using Corollary~\ref{cor:PWM}, we have the conclusion.
\end{proof}

\begin{prop}\label{prop:RT}
Let $q>1$, $-n/q\le\delta<0<\alpha<1$, $-n/q\le\beta<0$,
and
\begin{equation*}
 \psi(x,r)=
 \begin{cases}
  r^{\alpha}, & |x|\le2,\  0<r\le2, \\
  r^{\beta}, & |x|\le2,\ r>2, \\
  r^{\delta}, & |x|>2,\ 0<r\le2, \\
  r^{\beta}, & |x|>2,\ r>2.
 \end{cases}
\end{equation*}
Then
the Riesz transforms $\tR_j$, $j=1,2,\cdots,n$, are bounded
on $\cL_{q,\psi}(\R^n)$ and on $\cLN_{q,\psi}(\R^n)$.
That is,
there exists a positive constant $C$ such that,
for all $f\in\cL_{q,\psi}(\R^n)$,
$$
 \|\tR_jf\|_{\cL_{q,\psi}}\le C\|f\|_{\cL_{q,\psi}},
 \quad
 \|\tR_jf\|_{\cLN_{q,\psi}}\le C\|f\|_{\cLN_{q,\psi}},
 \quad
 j=1,2,\cdots,n.
$$
Moreover, 
if $f\in\cLN_{q,\psi}(\R^n)$ and $\sigma(f)=\lim_{r\to\infty}f_{B(0,r)}=0$,
then the Riesz transforms $R_jf$, $j=1,2,\cdots,n$, are well defined, 
$\sigma(R_jf)=\lim_{r\to\infty}(R_jf)_{B(0,r)}=0$, and
$$
 \|R_jf\|_{\cLN_{q,\psi}} \le C\|f\|_{\cLN_{q,\psi}},
 \quad
 j=1,2,\cdots,n.
$$
\end{prop}

\begin{proof}
We see that $\psi$ satisfies \eqref{phi-double} and 
$$
 r\int_r^{\infty}\frac{\psi(x,t)}{t^2}\,dt \le A\psi(x,r),
$$
for all $x\in\R^n$ and $r>0$.
Then we have the boundedness of $\tR_j$ 
on $\cL_{q,\psi}(\R^n)$ and on $\cLN_{q,\psi}(\R^n)$.
Let
\begin{equation*}
 \tpsi(x,r)=\tpsi(r)=
 \begin{cases}
  r^{\delta}, & 0<r\le2,  \\
  r^{\beta}, & r>2.
 \end{cases}
\end{equation*}
Then $\tpsi$ satisfies 
\eqref{phi-double}, \eqref{phi-near}, \eqref{int^infty phi}
and $\psi\le\tpsi$.
Therefore, by Theorem~\ref{thm:Cam-RT}, we have the conclusion.
\end{proof}

\section*{Acknowledgments}
The first author was partially supported by Grant-in-Aid for Scientific Research
(C), No.~24540159, Japan Society for the Promotion of Science.
The second author was partially supported by 
Grant-in-Aid for Young Scientists (B), No.~25870004,
Japan Society for the Promotion of Science.


\end{document}